\documentclass[11pt]{article}

\usepackage[a4paper,margin=1in]{geometry}
\usepackage[T1]{fontenc}
\usepackage[utf8]{inputenc}
\usepackage{lmodern}
\usepackage{microtype}
\usepackage{amsmath,amssymb,amsthm,mathtools}
\usepackage{array,booktabs,longtable}
\usepackage{enumitem}
\usepackage{xcolor}
\usepackage{listings}
\usepackage{fancyvrb}
\usepackage[hidelinks]{hyperref}

\newtheorem{theorem}{Theorem}[section]
\newtheorem{proposition}[theorem]{Proposition}

\theoremstyle{definition}

\newtheorem{remark}[theorem]{Remark}

\newcommand{\Sym}{\operatorname{Sym}}
\newcommand{\GL}{\operatorname{GL}}

\lstdefinestyle{pythonstyle}{
	language=Python,
	basicstyle=\ttfamily\scriptsize,
	breaklines=true,
	breakatwhitespace=false,
	showstringspaces=false,
	columns=fullflexible,
	frame=single,
	numbers=left,
	numberstyle=\tiny,
	tabsize=4,
	captionpos=b
}

\title{Counterexamples to Norm Conjectures of Wehlau in Modular Invariant Theory}

\author{Muhammad Fazeel Anwar\\
	\small Department of Mathematics, Sukkur IBA University, Sukkur, Pakistan. \\ fazeel.anwar@iba-suk.edu.pk}

\begin{document}

\maketitle

\begin{abstract}
	Let $G$ be a finite group and let $V$ be a finite-dimensional $G$-module over a field $k$. We construct explicit counterexamples in characteristic $2$ to conjectures of Wehlau concerning norms. We give faithful four-dimensional representations of $C_2^3$ and $C_2^4$ over $\mathbb F_8$ for which every nonlinear orbit norm is decomposable; in the $C_2^4$ example the invariant ring is polynomial, thereby disproving both Wehlau's norm conjecture and its unrestricted extension. 
	
\end{abstract}

\section{Introduction}
	Let $V$ be a finite-dimensional vector space over a field $k$ and let $G$ be a finite subgroup of $\GL(V)$. We choose a basis $\{z_0,\ldots,z_n\}$ for the dual space $V^*$. The action of $G$ on $V$ induces a natural action on $V^*$, which extends to an action on the symmetric algebra $k[V]=\Sym(V^*)=k[z_0,\ldots,z_n]$. We denote the ring of invariants by $k[V]^G$. A representation of $G$ is said to be modular if the characteristic of $k$ divides $|G|$. For a finitely generated graded algebra $R=\bigoplus_{i=0}^{\infty}R_i$, we denote by $R_{+}=\bigoplus_{i>0}R_i$ its ideal of positive-degree elements. An element of $R$ is called decomposable if it lies in $(R_{+})^2$.

	For a linear form $\ell\in V^*$, we distinguish the orbit norm and the full group norm $N_G^{\mathrm{orb}}(\ell)=\prod_{m\in G\ell}m, \, N_G^{\mathrm{full}}(\ell)=\prod_{g\in G}g(\ell)$. Thus $N_G^{\mathrm{full}}(\ell)=\bigl(N_G^{\mathrm{orb}}(\ell)\bigr)^{|\operatorname{Stab}_G(\ell)|}$. Following Wehlau's terminology, the orbit norm is nonlinear precisely when $\ell$ is not fixed by $G$. We show that both Wehlau's polynomial-ring conjecture and its stronger unrestricted form are false. The first example is a faithful four-dimensional representation of $C_2^3$ over $\mathbb F_8$ whose invariant ring has minimal generator degrees $1,1,4,4,6$, whereas every nonlinear orbit norm has degree $8$. More strongly, we construct a faithful four-dimensional representation of $C_2^4$ over $\mathbb F_8$ whose invariant ring is polynomial with generator degrees $1,1,4,4$, but for which every nonlinear orbit norm has degree $8$ or $16$. We also give counterexamples for the full group norm, including a family in every prime characteristic.

	Norms and orbit products are standard tools in modular invariant theory; a general account is given by Campbell and Wehlau \cite{CW}. The invariant rings of modular representations of elementary abelian $p$-groups were studied systematically by Campbell, Shank, and Wehlau \cite{CSW}, while the Klein four group in characteristic $2$ was treated by Sezer and Shank \cite{SS}. Polynomial invariant rings in positive characteristic, particularly for groups generated by pseudoreflections, were studied by Nakajima \cite{Nak} and Kemper and Malle \cite{KM}. The conjectures considered here were stated by Wehlau \cite{OP,WC}. Our main results are the following.

	\begin{theorem}
		\label{thm:norm-unrestricted}
		Let $k=\mathbb F_8=\mathbb F_2(\alpha)$, where
		$\alpha^3=\alpha+1$, and let
		$G=\langle\sigma,\tau,\rho\rangle\cong C_2^3$ act faithfully on $S=k[x_1,x_2,y_1,y_2]$ by fixing $x_1,x_2$ and by
		\[
		\begin{aligned}
			\sigma(y_1)&=y_1+x_1,
			&\sigma(y_2)&=y_2+x_2,\\
			\tau(y_1)&=y_1+\alpha x_2,
			&\tau(y_2)&=y_2+x_1,\\
			\rho(y_1)&=y_1+\alpha^2x_1,
			&\rho(y_2)&=y_2+(\alpha+\alpha^2)x_2.
		\end{aligned}
		\]
		Then $S^G$ has a minimal homogeneous algebra-generating set of degrees $1,1,4,4,6$. Every nonfixed linear form has orbit size $8$.  Consequently every nonlinear orbit norm has degree $8$ and belongs to $(S^G_+)^2$. Hence no nonlinear orbit norm can occur in a homogeneous minimal algebra generating set of $S^G$, and the unrestricted form of Wehlau's norm conjecture is false.
	\end{theorem}
	
	\begin{theorem}
		\label{thm:norm-polynomial}
		Let $k=\mathbb F_8=\mathbb F_2(\alpha)$, where
		$\alpha^3=\alpha+1$, and let
		$G=\langle\sigma_1,\sigma_2,\sigma_3,\sigma_4\rangle$ act faithfully on $S=k[x_1,x_2,y_1,y_2]$ by fixing $x_1,x_2$ and by
		\[
		\begin{aligned}
			\sigma_1(y_1)&=y_1+x_1,
			&\sigma_1(y_2)&=y_2,\\
			\sigma_2(y_1)&=y_1,
			&\sigma_2(y_2)&=y_2+x_2,\\
			\sigma_3(y_1)&=y_1+(1+\alpha^2)x_1+x_2,
			&\sigma_3(y_2)&=y_2+\alpha x_1+\alpha^2x_2,\\
			\sigma_4(y_1)&=y_1+(1+\alpha^2)x_1+\alpha^2x_2,
			&\sigma_4(y_2)&=y_2+(1+\alpha^2)x_1+\alpha^2x_2.
		\end{aligned}
		\]
		Then $G\cong C_2^4$ and $S^G$ is a polynomial algebra with minimal homogeneous generator degrees $1,1,4,4$. Every nonfixed linear form has orbit size $8$ or $16$.  Consequently
		every nonlinear orbit norm belongs to $(S^G_+)^2$, so no nonlinear orbit norm can occur in a homogeneous minimal algebra generating set.  Thus Wehlau's norm conjecture is false even when the invariant ring is polynomial.
	\end{theorem}

	\begin{theorem}\label{thm:norm-counterexample}
		Let $k=\mathbb F_2$, and let $G=C_2\times C_2=\langle\sigma,\tau\rangle$ act faithfully on $R=k[x,y,z]$ by
		\[
		\begin{array}{c|ccc}
			&x&y&z\\ \hline
			\sigma&x&y+x&z\\
			\tau&x&y&z+x.
		\end{array}
		\]
		Then $R^G=k[x,y^2+xy,z^2+xz]$ is a polynomial algebra, but $N_G^{\mathrm{full}}(\ell)\in(R^G_+)^2$ for every nonzero linear form $\ell\in R_1$. Hence no full group norm of a linear form occurs in a homogeneous minimal algebra-generating set of $R^G$, and Wehlau's conjecture is false for the full group norm.
	\end{theorem}


	The remainder of the paper proves these results and records their consequences.

\section{Proof of Theorem~\ref{thm:norm-unrestricted}}
\label{sec:unrestricted-norm}

Put $k=\mathbb F_8=\mathbb F_2(\alpha)$, where $\alpha^3=\alpha+1$, and let $R=k[x_1,x_2,y_1,y_2]^G$.  The action on the pair $(y_1,y_2)$ is given by translations whose shift matrices are $I_2,\,
T=\begin{pmatrix}0&\alpha\\ 1&0\end{pmatrix},\,
U=\begin{pmatrix}\alpha^2&0\\0&\alpha+\alpha^2\end{pmatrix}$. Every nonzero $\mathbb F_2$-linear combination of $I_2,T,U$ is
invertible.  Indeed, the determinants of the seven nonzero
combinations are precisely the seven nonzero elements of $k$.  The
three transformations are commuting involutions and their shift
matrices are $\mathbb F_2$-linearly independent.  Hence they generate a faithful group isomorphic to $C_2^3$.

\begin{proposition}
	\label{prop:unrestricted-ring}
	The invariant ring $R$ is minimally generated by five homogeneous
	invariants of degrees $1,1,4,4,6$.
\end{proposition}

\begin{proof}
	Set $f_1=x_1$, $f_2=x_2$, and define
	\[
	\begin{aligned}
		f_3={}&y_2^4+(1+\alpha^2)x_2^2y_2^2
		+(1+\alpha)x_2^2y_1^2+\alpha^2x_2^3y_2
		+(\alpha+\alpha^2)x_2^3y_1\\
		&+(1+\alpha)x_1x_2y_2^2+x_1x_2y_1^2
		+(1+\alpha^2)x_1x_2^2y_2
		+(1+\alpha)x_1^2x_2y_2+x_1^3y_2,
	\end{aligned}
	\]
	\[
	\begin{aligned}
		f_4={}&(1+\alpha^2)y_1^4+\alpha^2x_2^3y_1
		+\alpha x_1x_2y_2^2+(\alpha+\alpha^2)x_1x_2^2y_2\\
		&+\alpha^2x_1^2y_1^2+x_1^3y_1,
	\end{aligned}
	\]
	and
	\[
	\begin{aligned}
		f_5={}&x_2^2y_1^4+\alpha^2x_2^4y_1^2
		+(1+\alpha)x_1x_2y_2^4+\alpha^2x_1x_2^3y_2^2\\
		&+(1+\alpha^2)x_1x_2^3y_1^2
		+(1+\alpha+\alpha^2)x_1x_2^4y_2+x_1x_2^4y_1\\
		&+x_1^2y_2^4+\alpha^2x_1^2x_2^3y_2
		+(\alpha+\alpha^2)x_1^2x_2^3y_1\\
		&+(1+\alpha)x_1^3x_2y_2^2+x_1^4y_2^2.
	\end{aligned}
	\]
	Direct substitution gives $g(f_i)=f_i$ for every $g\in G$ and $1\leq i\leq5$.  Thus $A:=k[f_1,f_2,f_3,f_4,f_5]\subseteq R$. The four invariants $f_1,f_2,f_3,f_4$ form a homogeneous system of
	parameters.  Indeed, after setting $x_1=x_2=0$, the last two reduce to $f_3=y_2^4,\, f_4=(1+\alpha^2)y_1^4$, so their only common zero over an algebraic closure of $k$ is the origin.  Their degrees are $1,1,4,4$.  By Symonds' bound for secondary
	invariants \cite{Sy}, $R$ is generated as a module over $k[f_1,f_2,f_3,f_4]$ by homogeneous elements of degree at most $(1-1)+(1-1)+(4-1)+(4-1)=6$. Exact row reduction over $k$ gives
	\[
	\begin{array}{c|rrrrrrr}
		d&0&1&2&3&4&5&6\\ \hline
		\dim_k R_d&1&2&3&4&7&10&14\\
		\dim_k A_d&1&2&3&4&7&10&14.
	\end{array}
	\]
	Since $A_d\subseteq R_d$, it follows that $A_d=R_d$ for
	$0\leq d\leq6$.  Every required module generator therefore belongs to $A$, and hence $R=A$. Finally, exact row reduction gives
	\[
	\begin{array}{c|rrrrrr}
		d&1&2&3&4&5&6\\ \hline
		\dim_k\bigl(R_d/(R_+^2)_d\bigr)&2&0&0&2&0&1.
	\end{array}
	\]
	The images of $f_1,f_2$, of $f_3,f_4$, and of $f_5$ form bases of the
	nonzero indecomposable quotients in degrees $1,4,6$, respectively.
	Thus the displayed generating set is minimal, with degrees
	$1,1,4,4,6$.
\end{proof}

We now determine the orbits of linear forms.  Write $\ell=a_1x_1+a_2x_2+b_1y_1+b_2y_2$. If $(b_1,b_2)=(0,0)$, then $\ell$ is fixed.  Otherwise, a nonidentity element of $G$ can fix $\ell$ only if the transpose of its nonzero shift matrix annihilates $(b_1,b_2)^T$.  This is impossible because every nonzero shift matrix is invertible.  Therefore $\operatorname{Stab}_G(\ell)=1$ and $|G\ell|=8$ for every nonfixed linear form. It follows that every nonlinear orbit norm $N_G^{\mathrm{orb}}(\ell)=\prod_{m\in G\ell}m$ has degree $8$.  Proposition~\ref{prop:unrestricted-ring} shows that $R$ is generated in degrees $1,4,6$, and therefore $R_8=(R_+^2)_8$.  Hence every nonlinear orbit norm is decomposable and no such norm can occur in a homogeneous minimal algebra-generating set. This proves Theorem~\ref{thm:norm-unrestricted}.

\begin{remark}
	Every nonfixed linear form in Theorem~\ref{thm:norm-unrestricted} has a regular orbit, yet its orbit norm is decomposable. Thus the existence of a regular orbit is not sufficient for an orbit norm to occur in a homogeneous minimal algebra-generating set.
\end{remark}

\section{Proof of Theorem~\ref{thm:norm-polynomial}}
\label{sec:polynomial-norm}

Put $k=\mathbb F_8=\mathbb F_2(\alpha)$, where
$\alpha^3=\alpha+1$, and let
$S=k[x_1,x_2,y_1,y_2]$.  The four generators have shift matrices $B_1=\begin{pmatrix}1&0\\0&0\end{pmatrix},\,
B_2=\begin{pmatrix}0&0\\0&1\end{pmatrix}, \, B_3=\begin{pmatrix}1+\alpha^2&1\\ \alpha&\alpha^2\end{pmatrix},
\, B_4=\begin{pmatrix}1+\alpha^2&\alpha^2\\
	1+\alpha^2&\alpha^2\end{pmatrix}$. Each $B_i$ has rank one, so each $\sigma_i$ is a transvection.  The four matrices are linearly independent over $\mathbb F_2$.  Since the corresponding translations commute and have order two, they generate a faithful group $G\cong C_2^4$. Define
\[
\begin{aligned}
	F_1={}&y_1^4+(1+\alpha)x_2^2y_2^2
	+(1+\alpha+\alpha^2)x_2^2y_1^2+(1+\alpha)x_2^3y_2\\
	&+(\alpha+\alpha^2)x_1x_2y_2^2
	+(\alpha+\alpha^2)x_1x_2y_1^2
	+(\alpha+\alpha^2)x_1x_2^2y_2\\
	&+(1+\alpha+\alpha^2)x_1x_2^2y_1
	+(1+\alpha)x_1^2y_1^2
	+(\alpha+\alpha^2)x_1^2x_2y_1+\alpha x_1^3y_1,
\end{aligned}
\]
and
\[
\begin{aligned}
	F_2={}&y_2^4+(1+\alpha)x_2^2y_2^2+\alpha x_2^3y_2
	+(1+\alpha)x_1x_2y_2^2+\alpha x_1x_2y_1^2\\
	&+(1+\alpha)x_1x_2^2y_2+(1+\alpha)x_1^2y_2^2
	+(1+\alpha^2)x_1^2y_1^2\\
	&+(1+\alpha)x_1^2x_2y_2+\alpha x_1^2x_2y_1
	+(1+\alpha^2)x_1^3y_1.
\end{aligned}
\]
Direct substitution gives $\sigma_i(F_j)=F_j$ for
$1\leq i\leq4$ and $j=1,2$.  Hence $A:=k[x_1,x_2,F_1,F_2]\subseteq S^G$. The four displayed invariants have only the origin as a common zero over an algebraic closure of $k$: after $x_1=x_2=0$, the final two equations become $y_1^4=y_2^4=0$.  They therefore form a homogeneous system of parameters and are algebraically independent.  Since $S$ is a polynomial ring, it is a free $A$-module of rank
$1\cdot1\cdot4\cdot4=16$. Consequently $[\operatorname{Frac}(S):\operatorname{Frac}(A)]=16$.
The action is faithful, so $[\operatorname{Frac}(S):\operatorname{Frac}(S^G)]=|G|=16$. As $A\subseteq S^G$, the two invariant fraction fields are equal. Moreover, $S^G$ is integral over $A$, while $A$ is a polynomial ring
and hence integrally closed.  Therefore $S^G=A=k[x_1,x_2,F_1,F_2]$.
In particular, $S^G$ is polynomial with minimal homogeneous generator
degrees $1,1,4,4$.

It remains to examine the norms.  Write $\ell=c_1x_1+c_2x_2+b_1y_1+b_2y_2$. If $(b_1,b_2)=(0,0)$, then $\ell$ is fixed.  For $b=(b_1,b_2)^T\neq0$, let $\mathcal B=\langle B_1,B_2,B_3,B_4\rangle_{\mathbb F_2}$ and consider the $\mathbb F_2$-linear map $\phi_b:\mathcal B\longrightarrow k^2,\qquad B\longmapsto B^Tb$. The stabilizer of $\ell$ is the kernel of $\phi_b$, and hence $|G\ell|=2^{\operatorname{rank}_{\mathbb F_2}(\phi_b)}$. There are nine projective directions for $b$.  Exact row reduction
gives
\[
\begin{array}{c|c}
	[b_1:b_2]&\operatorname{rank}_{\mathbb F_2}(\phi_b)\\ \hline
	[0:1], [1:0], [1:1], [1:\alpha], [1:1+\alpha],
	[1:1+\alpha+\alpha^2]&3,\\
	{}[1:\alpha^2], [1:1+\alpha^2], [1:\alpha+\alpha^2]&4.
\end{array}
\]
Thus every nonfixed linear form has orbit size $8$ or $16$.  More
precisely, the projective orbit size distribution is $1^9,\, 8^{48},\, 16^{12}$, that is, there are nine orbits of size 1, forty-eight orbits of size 8, and twelve orbits of size 16. Therefore every nonlinear orbit norm has degree $8$ or $16$.  Since $S^G=k[x_1,x_2,F_1,F_2],\, \deg(x_1,x_2,F_1,F_2)=(1,1,4,4)$,
every invariant of degree $8$ or $16$ is contained in $(S^G_+)^2$.
In particular, $N_G^{\mathrm{orb}}(\ell)\in(S^G_+)^2$ for every nonfixed linear form $\ell$. This proves Theorem~\ref{thm:norm-polynomial}.

\section{Proof of Theorem~\ref{thm:norm-counterexample}}
\label{sec:norm-counterexample}

Throughout this section we use the full group norm $N_G^{\mathrm{full}}(\ell)=\prod_{g\in G}g(\ell)$. Let $k=\mathbb F_2$, and let $G=C_2\times C_2=\langle\sigma,\tau\rangle$
act on $R=k[x,y,z]$ by
\[
\begin{array}{c|ccc}
	&x&y&z\\ \hline
	\sigma&x&y+x&z\\
	\tau&x&y&z+x.
\end{array}
\]

	The transformations \(\sigma\) and \(\tau\) are commuting involutions,
	and the four elements $1,\, \sigma,\, \tau,\, \sigma\tau$ act distinctly. Hence the action of \(G\) is faithful. Put $a=y^2+xy,\, b=z^2+xz$. Since the characteristic is \(2\), $\sigma(a)=(y+x)^2+x(y+x)=y^2+xy=a$. Thus $k[x,z,a]\subseteq R^{\langle\sigma\rangle}$. Using the relation $y^2=a+xy$, every element of \(R\) may be written in the form $u+yv,\, u,v\in k[x,z,a]$. This expression is unique. Indeed, if \(u+yv=0\), then applying \(\sigma\) and subtracting the original equality gives $xv=0$.
	Since \(R\) is a domain, \(v=0\), and then \(u=0\). Therefore $R=k[x,z,a]\oplus yk[x,z,a]$. Also $\sigma(u+yv)=u+(y+x)v$, so \(u+yv\) is fixed by \(\sigma\) if and only if \(xv=0\), equivalently \(v=0\). Hence
	$R^{\langle\sigma\rangle}=k[x,z,a]$. The element \(\tau\) fixes \(x\) and \(a\), and sends \(z\) to \(z+x\). Using $z^2=b+xz$, the same argument gives $k[x,z,a]=k[x,a,b]\oplus zk[x,a,b]$ and $\bigl(k[x,z,a]\bigr)^{\langle\tau\rangle} = k[x,a,b]$. Consequently,
	\[
	R^G=k[x,a,b]
	=k[x,y^2+xy,z^2+xz].
	\]
	
	Moreover, \(R\) is integral over \(k[x,a,b]\), since \(y\) and \(z\)
	satisfy the monic equations $T^2+xT+a=0, \, T^2+xT+b=0$, respectively. It follows that \(x,a,b\) are algebraically independent, and hence \(R^G\) is a polynomial algebra. Let $\ell=\alpha x+\beta y+\gamma z, \, \alpha,\beta,\gamma\in\mathbb F _2$. Then $\sigma(\ell)=\ell+\beta x,\,
	\tau(\ell)=\ell+\gamma x,\, \sigma\tau(\ell)=\ell+(\beta+\gamma)x$. Every nonzero linear form has a nontrivial stabilizer. Indeed,
	\[
	\begin{cases}
		\sigma\in\operatorname{Stab}_G(\ell),&\beta=0,\\
		\tau\in\operatorname{Stab}_G(\ell),&\gamma=0,\\
		\sigma\tau\in\operatorname{Stab}_G(\ell),&\beta=\gamma=1.
	\end{cases}
	\]
	Thus, for $H=\operatorname{Stab}_G(\ell)$, we have \(|H|\geq2\). Let $q_\ell = \prod_{gH\in G/H}g(\ell)$ be the product of the distinct elements in the orbit of \(\ell\). Then \(q_\ell\in R^G_+\), and every orbit element occurs \(|H|\) times in the full group norm. Therefore $N_G^{\mathrm{full}}(\ell)=\prod_{g\in G}g(\ell)=q_\ell^{\,|H|}\in (R^G_+)^2$. The images of the elements in a homogeneous minimal algebra-generating
	set of a connected graded algebra form a basis of $R^G_+/(R^G_+)^2$. Every norm has zero image in this quotient, so no norm of a linear form can occur in any homogeneous minimal algebra-generating set of \(R^G\).

\begin{remark}
	For the standard linear forms, one has $N_G^{\mathrm{full}}(x)=x^4$, $N_G^{\mathrm{full}}(y)=y^2(y+x)^2=(y^2+xy)^2=a^2$, and $N_G^{\mathrm{full}}(z)=z^2(z+x)^2=(z^2+xz)^2=b^2$.
	Similarly, $N_G^{\mathrm{full}}(y+z)=\bigl((y+z)(y+z+x)\bigr)^2$. Thus these norms are visibly decomposable in \(R^G\).
\end{remark}

\begin{remark}
	The proof gives a general obstruction. If a nonzero linear form $\ell$ has a nontrivial stabilizer $H\leq G$, then $N_G^{\mathrm{full}}(\ell)=
	\left(\prod_{gH\in G/H}g(\ell)\right)^{|H|}
	\in (R^G_+)^2$. Consequently, an indecomposable full group norm can occur only when the action of $G$ on the space of linear forms has a regular orbit.
\end{remark}

\begin{remark}
	The distinction between the full group norm and the orbit norm is essential. For example $\prod_{g\ell\in G\cdot y}g\ell =y(y+x)=a$, and similarly $N_G^{\mathrm{orb}}(z)=b$. These are
	minimal generators of \(R^G\). The counterexample in this section concerns the full group norm $N_G^{\mathrm{full}}(\ell)=\prod_{g\in G}g(\ell)$, rather than the orbit norm.
\end{remark}


\begin{thebibliography}{25}

\bibitem{CW} H.~E.~A.~Eddy Campbell and D.~L.~Wehlau, Modular Invariant Theory, \emph{Encyclopaedia of Mathematical Sciences}, vol.~139, Springer, Berlin, 2011, doi:10.1007/978-3-642-17404-9.

\bibitem{OP} D.~L.~Wehlau, Some problems in invariant theory, in Invariant Theory in All Characteristics, \emph{CRM Proc. Lecture Notes}, vol.~35, Amer. Math. Soc., Providence, RI, 2004, pp.~265--274, doi:10.1090/crmp/035/20.

\bibitem{Sy} P.~Symonds, On the Castelnuovo--Mumford regularity of rings of polynomial invariants, \emph{Ann. of Math. (2)} \textbf{174} (2011), no.~1, 499--517, doi:10.4007/annals.2011.174.1.14.

\bibitem{WC} D.~L.~Wehlau, Some conjectures of mine and others, available at \url{https://mast.queensu.ca/~wehlau/Conjectures.html}, accessed July 2026.

\bibitem{CSW} H.~E.~A.~Campbell, R.~J.~Shank, and D.~L.~Wehlau, Rings of invariants for modular representations of elementary abelian $p$-groups, \emph{Transform. Groups} \textbf{18} (2013), no.~1, 1--22, doi:10.1007/s00031-013-9207-z.

\bibitem{KM} G.~Kemper and G.~Malle, The finite irreducible linear groups with polynomial ring of invariants, \emph{Transform. Groups} \textbf{2} (1997), no.~1, 57--89, doi:10.1007/BF01234631.

\bibitem{Nak} H.~Nakajima, Invariants of finite groups generated by pseudo-reflections in positive characteristic, \emph{Tsukuba J. Math.} \textbf{3} (1979), no.~1, 109--122, doi:10.21099/tkbjm/1496158618.

\bibitem{SS} M.~Sezer and R.~J.~Shank, Rings of invariants for modular representations of the Klein four group, \emph{Trans. Amer. Math. Soc.} \textbf{368} (2016), no.~8, 5655--5673, doi:10.1090/tran/6516.
\end{thebibliography}
\end{document}